\documentclass[11pt]{amsart}
\usepackage{amsmath}
\usepackage{}
\usepackage{amsfonts}
\usepackage{amsfonts,latexsym,rawfonts,amsmath,amssymb,amsthm}
\usepackage[plainpages=false]{hyperref}

\usepackage{graphicx}
\makeatletter
\renewcommand\@biblabel[1]{}
\makeatother

\numberwithin{equation}{section}

\newcommand{\beq}{\begin{equation}}
\newcommand{\eeq}{\end{equation}}
\newcommand{\beqs}{\begin{eqnarray*}}
\newcommand{\eeqs}{\end{eqnarray*}}
\newcommand{\beqn}{\begin{eqnarray}}
\newcommand{\eeqn}{\end{eqnarray}}
\newcommand{\beqa}{\begin{array}}
\newcommand{\eeqa}{\end{array}}

\def\lra{\longrightarrow}

\def\bc{\begin{center}}
\def\ec{\end{center}}

\def\begeq{\begin{equation}}
\def\endeq{\end{equation}}
\def\and{\quad{\rm and}\quad}

\let\lra=\longrightarrow

\def\mapright\#1{\,\smash{\mathop{\lra}\limits^{\#1}}\,}

\newtheorem{prop}{Proposition}[section]
\newtheorem{theo}[prop]{Theorem}
\newtheorem{lem}[prop]{Lemma}
\newtheorem{claim}[prop]{Claim}
\newtheorem{cor}[prop]{Corollary}
\newtheorem{rem}[prop]{Remark}

\begin{document}

\date{}
\author {Yuxing Deng }
\author { Xiaohua $\text{Zhu}^*$}

\thanks {* Partially supported by the NSFC Grants 11271022 and 11331001}

\address{ Yuxing Deng\\School of Mathematical Sciences, Beijing Normal University,
Beijing, 100875, China\\
dengyuxing@mail.bnu.edu.cn}

\address{ Xiaohua Zhu\\School of Mathematical Sciences and BICMR, Peking University,
Beijing, 100871, China\\
xhzhu@math.pku.edu.cn}

\title{ Asymptotic behavior of positively curved steady Ricci Solitons, II}
\maketitle

\section*{\ }

\begin{abstract} In this note,  we show  that any $n$-dimensional $\kappa$-noncollapsed steady K\"ahler-Ricci soliton with nonnegative bisectional curvature must be flat.  The result is an  improvement to  our former  work in  \cite{DZ2}.
\end{abstract}

\section{Introduction}

In    \cite{DZ2}, we prove

\begin{theo}\label{main-theorem-nonexistence-1-old} There is no any
$\kappa$-noncollapsed steady K\"{a}hler-Ricci soliton with   nonnegative sectional curvature and positive Ricci curvature.
\end{theo}

In this note,   we  improve Theorem \ref{main-theorem-nonexistence-1-old}  in sense of  nonnegative bisectional curvature.  Namely, we have

\begin{theo}\label{main-theorem-nonexistence-1}There is no any
$\kappa$-noncollapsed steady K\"{a}hler-Ricci soliton with   nonnegative bisectional curvature and positive Ricci curvature.
\end{theo}

As an application of Theorem  \ref{main-theorem-nonexistence-1-old}, we  get the following rigidity result for the steady K\"ahler-Ricci solitons and  K\"{a}hler-Ricci flow.

\begin{theo}\label{theorem-eternal-flow}
Any $\kappa$-noncollapsed steady K\"ahler-Ricci soliton $(M,g,f)$  with nonnegative bisectional curvature must be flat. More generally, any $\kappa$-noncollapsed noncompact and eternal K\"{a}hler-Ricci flow with  nonnegative bisectional curvature  and uniformly bounded curvature must be a flat flow. 
\end{theo}

In the proof of Theorem \ref{main-theorem-nonexistence-1-old}, one main step is to use the blow-down argument to analysis the structure of limit  K\"{a}hler-Ricci flow of a sequence of rescaled  K\"{a}hler-Ricci flows.  The  nonnegative sectional curvature condition is used  to guarantee  the existence  of lines  on  the  limit   flow  by using Toponogov comparison Theorem (cf.  \cite{MT}) and consequently   the  limit   flow can be split off a  line by Cheeger-Gromoll splitting theorem \cite{CG}.  The splitting property  is crucial in  our proving the curvature decay:
\begin{align}\label{curvature-decay}
R(p,t)\rightarrow0,~ as ~t\rightarrow-\infty.
\end{align}
Here $R(p,t)$ is a scalar curvature of   corresponding Ricci flow of   steady K\"ahler-Ricci soliton  $(M,g,f)$. 

The new ingredient at present  is to prove a local  splitting result for $\kappa$-noncollapsed steady K\"ahler-Ricci solitons with  nonnegative bisectional curvature (cf.  Lemma \ref{lem-curvature-decay}).    Then we  can generalize the argument in  the proof of Theorem \ref{main-theorem-nonexistence-1-old} to  Theorem \ref{main-theorem-nonexistence-1}.

Theorem \ref{theorem-eternal-flow} can be regarded as a generalization of Ni's rigidity  theorem  for ancient solution of  K\"{a}hler-Ricci flow with maximal volume growth \cite{Ni}. Ni's result is a complex version  of Perelman's  Theorem for ancient solution of Ricci flow in  \cite{Pe}, Section 11.4.

\vskip5mm

\noindent {\bf Acknowledgements.} The work was done partially when the second author is   visiting at the Mathematical Sciences Research Institute at Berkeley during the spring 2016 semester. 
 The author  would like to thank her hospitality and the financial supports,    National Sciences Foundation under Grant No. DMS-1440140, and  Simons Foundation.

\section{Primary results in \cite{DZ2}}

In this section, we recall some results proved in \cite{DZ2}, which will be used  to complete  the proof of Theorem  \ref{main-theorem-nonexistence-1} in the  next section.

Let $(M,g,f)$ be a complete $\kappa$-noncollapsed  steady K\"{a}hler-Ricci soltion with  nonnegative bisectional curvature.  It is proved that there exists a  quilibrium point  $o$ of $M$ such that $\nabla f(o)=0$ in \cite{DZ1}.  Let $\phi_{t}$ be a family of biholomorphisms generated by $-\nabla f$. Let $g(t)=\phi^{*}_{t}(g)$. Then $g(t)$ satisfies the Ricci flow
\begin{align}\label{ricci-equation} \frac{\partial g}{\partial t}=-2 {\rm Ric}(g),~t\in(-\infty,\infty).
\end{align}

Let  $p_{i}\in M$  be any sequence with ${\rm dist}(o,p_{i})\rightarrow\infty$ as $i\to\infty $. We consider a sequence of rescaled flows $(M,R(p_{i})g(R^{-1}(p_{i})t),p_{i})$. 
By using the compactness theorem, Theorem 3.3 in  \cite{DZ2},  we prove the following convergence result.
  
\begin{theo}\label{soliton-1}
$(M,R(p_{i})g(R^{-1}(p_{i})t),p_{i})$ converges subsequently to a pseudo $\kappa$-solution  $(M_{\infty},g_{\infty}(t))$   ($t\in (-\infty, \infty)$) of K\"{a}hler-Ricci flow in the Cheeger-Gromov topology.
\end{theo}

  Theorem \ref{soliton-1} is in  fact of a part of Theorem 1.5 in  \cite{DZ2}.   Since  $(M,g,f)$ is weakened  to have  
  nonnegative bisectional curvature condition, we could not get the splitting property of limit  $(M_{\infty},g_{\infty}(t))$ as in Theorem 1.5 under nonnegative sectional curvature condition. Here  a pseudo $\kappa$-solution  of K\"{a}hler-Ricci flow   means that a non-flat  K\"{a}hler-Ricci flow  with $\kappa$-noncollapsed condition satisfies the  Harnack inequality,
  \begin{equation}\label{eq:Harnack-complex}
\frac{\partial R}{\partial t}+\nabla_{i}RV^{i}+\nabla_{\bar{i}}RV^{\bar{i}}+R_{i\bar{j}}V^{i}V^{\bar{j}}\geq0,~\forall~V\in T^{(1,0)}M.
\end{equation}
   
    For e a steady K\"{a}hler-Ricci soliton  $(M,g, J)$ with positive Ricci curvature, which admits an equilibrium point,  Bryant  in \cite{Bry} proves  that
    there exist global holomorphic coordinates
(Poincar\'{e} coordinates)
$z:M\rightarrow \mathbb{C}^{n}$
which linearize  $Z=\frac{\nabla f-\sqrt{-1}J\nabla f}{2}$ such that
\begin{equation}\label{Bryant-eq}
Z=\sum_{i=1}^{n}h_{i}z_{i}\frac{\partial}{\partial z_{i}},
\end{equation}
  where  $h_{1},\cdots,h_{n}$   are positive constants.
 As in Corollary 5.2  in  \cite{DZ2}, we  choose points $p_{k}=(k,0,\cdots,0,0\cdots,0)\in M$.  By computation of lengths of     integral curves of $J\nabla f$ starting from $p_{k}$, we have

\begin{lem}\label{prop-contradiction argument}
Suppose that $(M,g,f)$ is  an $n$-dimensional $\kappa$-noncollapsed steady K\"ahler-Ricci soliton with nonnegative bisectional curvature and positive Ricci curvature.  Then  there exists a positive constant $C$ such that $R(p_{k})>C$, where $C$ is independent of $p_{k}$.
\end{lem}

 Lemma \ref{prop-contradiction argument} is in fact from  Lemma 5.4 in \cite{DZ2} while the nonnegative sectional curvature  condition is  replaced by  nonnegative bisectional curvature  condition.   This can be done since  we only use the  convergence result,  Theorem \ref{soliton-1} in the  proof  Lemma 5.4.

\section{Proof of  Theorem  \ref{main-theorem-nonexistence-1}}

The following proposition is a key lemma in the proof of Theorem   \ref{main-theorem-nonexistence-1}.

\begin{lem}\label{lem-curvature-decay}
Let $(M,g,f)$ be a $\kappa$-noncollapsed steady K\"{a}hler-Ricci soltion with nonnegative bisectional curvature and positive Ricci curvature. Then, for any $p\in M\setminus\{o\}$, at least one of the following two properties holds:

i)~$R(p,t)\rightarrow0$, as $t\rightarrow-\infty$.

ii)~For any $\tau_{i}\rightarrow -\infty$, $(M,R(p,\tau_i)g(R^{-1}(p,\tau_i)t),\phi_{\tau_{i}}(p))$ converges subsequently to a  flow of  $\kappa$-noncollapsed steady K\"{a}hler-Ricci soltions   $(M_{\infty},g_{\infty}(t),$
\newline $p_{\infty})$ with nonnegative bisectional curvature  
which  splits locally    off   a complex line.
\end{lem}

\begin{proof}First we note 
\begin{align}\label{inden} |\nabla f|^2+R=R_{\max}=\sup_{x\in M}R(x).\end{align}
and  $R(p)\geq 0$. Then  by the uniqueness of  quilibrium points, we have
$$ 0< R(p,t)<R_{\max}.$$  Thus by  relation
\begin{align}\label{montone}
\frac{\partial R(p,t)}{\partial t}=2{\rm Ric}( \nabla f, \nabla f)(\phi_t(p))>0,
\end{align}
we see that
  $\lim_{t\rightarrow-\infty}R(p,t)$ exists.  Now we suppose  that  the property i) in Lemma  \ref{lem-curvature-decay}  does not hold.  Then there exists  a point $p\in M$ such that
\begin{align}\label{eq:inf limit-of-R}
\lim_{t\rightarrow-\infty}R(p,t)=C_0>0.
\end{align}
Since $o$ is  an unique quilibrium point, by (\ref{montone}),
we have $C_0<R_{\max}$.

Consider any  sequence $(M,g_{i}(t),p_{\tau_i})$, where $g_{i}(t)=R(p,\tau_i)g(R^{-1}(p,\tau_i)t)$
 and $p_{\tau_i}=\phi_\tau(p)$  with $\tau_i\to - \infty$.
Then  $(M,g_{i}(t))$ is  $\kappa$-noncollapsed and   curvature of $(M,g_{i}(t))$ is uniformly bounded.  By Hamilton 
compactness theorem \cite{Ha},  $(M,g_{{i}}(t),p_{\tau_{i}})$    converges  subsequently to a  pseudo $\kappa$-solution $(M_{\infty},g_{\infty}(t),p_{\infty})$,  where $t\in(-\infty,\infty)$.  Moreover,  by  $(\ref{eq:inf limit-of-R})$, we have 
\begin{align}\label{eq:constant limit}
R_{\infty}(p_{\infty},t)=\lim_{\tau_{i}\rightarrow-\infty}\frac{ R(  p_{\tau_{i}},  R^{-1}(p,\tau_{i})t)}{R(p,\tau_{i})}=1,~\forall~t\in(-\infty,+\infty).
\end{align}
and consequently,
\begin{equation}\label{eq:2-1}
\frac{\partial}{\partial t}R_{\infty}(p_{\infty},t)\equiv0.
\end{equation}
 Since  $(M_{\infty},g_{\infty}(t);p_{\infty})$ is not flat  by $(\ref{eq:constant limit})$,  we may assume that $(M_{\infty},g_{\infty}(t))$ has positive Ricci curvature   by
Cao's dimension reduction theorem \cite{C2}.
By  the Harnack inequality (\ref{eq:Harnack-complex}) together with condition  (\ref{eq:2-1}),   following the argument in the proof of Theorem 4.1 in \cite{C1}, we
can further prove that $(M_{\infty},g_{\infty}(t), p_{\infty})$ is in fact a steady K\"{a}hler-Ricci soliton, which is $\kappa$-noncollapsed and has nonnegative bisectional curvature and positive Ricci curvature (also see Proposition 2.2, \cite{CT}). More precisely,  there is a smooth vector field $V$ on $M_\infty$ such that
 \begin{align}\label{eq:2-2}
R^{(\infty)}_{ij}=\frac{1}{2}(\nabla^{(\infty)}_{i}V_{j}+\nabla^{(\infty)}_{j}V_{i})
\end{align}
and
\begin{align}\label{eq:2-3}
\nabla^{(\infty)}_{i}V_{j}=\nabla^{(\infty)}_{j}V_{i}.
\end{align}
Thus
$$\frac{\partial R^{(\infty)}(p_{\infty},t)}{\partial t}|_{t=0}=2{\rm Ric}^{(\infty)}(V,V).$$
By (\ref{eq:2-1}),  it follows  $V(p_{\infty})=0$.

Let $X_{(i)}=R^{-1}(p,\tau_i)\nabla f$. Then
$$|\nabla^{i}X_{(i)}|_{\hat g_i}\le  C_0^{-1} |{\rm Ric}^{(i)}|_{\hat g_i}\le C,$$
where $\hat g_i=R(p,\tau_i)g$ and $ {\rm Ric}^{(i)}$ is Ricci curvature of $\hat g_i$.  By Shi's higher order estimate \cite{Sh} and  the soliton equation, we also get
$$|(\nabla^{i})^m X_{(i)}|_{\hat g_{i}}\leq C(n)|( \nabla^{i})^{m-1} {\rm Ric}^{(i)}|_{\hat g_{i}}\le C(m).$$
Thus   by taking a subsequence, we may assume that
$$X_{(i)}\rightarrow X~as~i\rightarrow\infty.$$
Since
\begin{align*}
\nabla^{(i)}_j X_{(i)k}=\nabla^{(i)}_k X_{(i)j}=R^{(i)}_{jk},
\end{align*}
we get
\begin{align}\label{eq:2-4}
\nabla^{(\infty)}_j X_k=\nabla^{(\infty)}_k X_j=R^{(\infty)}_{jk}.
\end{align}
Moreover, by (\ref{inden}), 
$$
|X|_{g_{\infty}}(p_\infty)= \lim_{i\rightarrow\infty}|X_{(i)}|_{g_i}(p_i)=\lim_{i\rightarrow\infty}\sqrt{\frac{R_{\max}}{R(p,t_i)}-1}.
$$
Hence
\begin{align}\label{non-zero}
|X|_{g_{\infty}}(p_\infty)=\sqrt{C_0^{-1}R_{\max}-1}>0.
\end{align}

Let $W=V-X$. Then by  (\ref{eq:2-2}), (\ref{eq:2-3}) and (\ref{eq:2-4}), we have
\begin{align}\label{eq:2-5}
\nabla^{(\infty)}W=\nabla^{(\infty)}(V-X)\equiv0.
\end{align}
On the other hand,   $V(p_{\infty})=0$.  By (\ref{non-zero}),  we see that $W$ is nonzero everywhere.   Thus  the K\"{a}hler manifold  $(M_{\infty},g_{\infty}(0))$  splits locally  off a line along $W$. Note that $J_{\infty}W(p_{\infty})\neq0$ and $$\nabla^{(\infty)}(J_{\infty}W)=0,$$
where  $J_{\infty}$ is the complex structure of $M_\infty$.  Hence $(M_{\infty},g_{\infty}(0))$ also splits locally off a line  along $J_{\infty}W$.  As a  conseqence,  the steady soliton $(M_{\infty},g_{\infty}(0))$  splits locally  off  a  complex  line.  Therefore,  the property  ii)  in Lemma  \ref{lem-curvature-decay}  holds. The lemma is 
proved.
\end{proof}

\begin{rem}
Lemma  \ref{lem-curvature-decay} is still true for  the steady K\"{a}hler-Ricci soliton with   nonnegative bisectional curvature if the $\kappa$-noncollapsing   condition is replaced by the existence of uniform injective radius of soliton.
\end{rem}

\begin{cor}\label{cor-of-splitting-theorem} Let $(M,g,f)$ be a $2$-dimensional
$\kappa$-noncollapsed steady K\"{a}hler-Ricci soliton with nonnegative bisectional curvature and positive Ricci curvature.   Then for any  $p\neq o$, 
\begin{equation}\label{asymptotic-r}
R(p,t)\rightarrow0, ~ {\rm as} ~t\rightarrow-\infty.
\end{equation}

\end{cor}

\begin{proof} We prove the corollary  by contradiction.  Then  by  Lemma \ref{lem-curvature-decay}  there exist a point $p\in M$  and  a sequence of $t_i\to-\infty$ such that
\begin{align}\label{eq:inf limit of R}
\lim_{t\rightarrow-\infty}R(p,t_i)=C>0
\end{align}
and  rescaled $(M,R(p,t_i)g(R^{-1}(p,t_i)t),\phi_{t_{i}}(p))$ converges subsequently to  a  flow of $2$-dimensional   complete $\kappa$-noncollapsed,  steady K\"{a}hler-Ricci solitons   
\newline $(M_{\infty},g_{\infty}(t),p_{\infty})$ $(t\in (-\infty, \infty))$
which splits  locally  off  a complex line.    Moreover
\begin{align}\label{eq:constant limit-2}
R_{\infty}(p_{\infty},t)=1,~for~t\in(-\infty,+\infty).
\end{align}

 Let $(\widetilde {M}_{\infty}, \widetilde g( 0))$   be   the  universal covering of  $({M}_{\infty}, g( 0))$. Then it is easy to see that $(\widetilde {M}_{\infty}, \widetilde g( 0))$ still satisfies  the $\kappa$-noncollapsing   condition. Moveover, the parallel vector field $W$ in  Lemma \ref{lem-curvature-decay} can be lifted on $(\widetilde {M}_{\infty}, \widetilde g( 0))$ and this vector field  generates a trivial honolomy group of  $(\widetilde {M}_{\infty}, \widetilde g( 0))$. Thus by Wu's  de Rham decomposition Theorem \cite{Wu},   the  steady soliton $(\widetilde {M}_{\infty}, \widetilde g( 0))$  splits off  a  complex line.  Hence the  corresponding  steady solitons flow  $(\widetilde {M}_{\infty}, \widetilde g( t))$ splits off a  complex line, and so it  splits out a real $2$-dimensional   flow  $(N,g_N(t))$  of  complete   $\kappa$-noncollapsed steady Ricci solitons  with positive  curvature.   
  On the other hand, by Lemma 4.4 in   \cite{DZ2}, any  complete  pseudo-$\kappa$-solution on a surface is a shrinking flow of round spheres and so $(N,g_N(t))$ does.   In particular,  $(N, g_N(0))$ is compact.  But this is impossible since any compact  gradient steady Ricci soliton should be flat.  The corollary is proved.

\end{proof}

 Choosing  $p\neq o \in M$ and a  sequence of $t_k\rightarrow-\infty$.  Let $p_k=\phi_{t_k}$.   Then by Corollary \ref{cor-of-splitting-theorem},  $R(p,t_k)\rightarrow0 ~ {\rm as} ~t_k\rightarrow-\infty.$  Moreover, one can check 
  $p_{k}=( e^{-t_{k}h_1}, ...,0)$ under Poincar\'{e} coordinates.    Thus applying   Lemma \ref{prop-contradiction argument} to  $2$-dimensional  steady K\"ahler-Ricci soliton  together with Cao's dimension reduction theorem, we prove
  
\begin{prop}\label{flat-theorem}
Let $(M,g)$ be a $2$-dimensional $\kappa$-noncollapsed steady K\"ah-\\ler-Ricci soliton with nonnegative bisectional curvature.   Then  $(M,g)$  is flat.
\end{prop}

Now we complete  the proof of Theorem  \ref{main-theorem-nonexistence-1}.

\begin{proof}[Proof of Theorem \ref{main-theorem-nonexistence-1}]
By Proposition \ref{flat-theorem},  we can use the induction argument as in the proof of Theorem 1.3, \cite{DZ2}.  Suppose that there is no $k$-dimensional
$\kappa$-noncollapsed steady K\"{a}hler-Ricci soliton with nonnegative bisectional curvature and positive Ricci curvature for all $k<n$. We claim

\begin{claim}\label{key-lemma} Under the induction hypothesis, for any fixed $p\in M\setminus\{o\}$,
$R(p,-t)\to 0$ as $t\to\infty.$
\end{claim}

 The proof  of   Claim \ref{key-lemma}  is  similar to one of Corollary \ref{cor-of-splitting-theorem}.    In fact, if the claim is not true, then  by  Lemma \ref{lem-curvature-decay}  there exist a point $p\in M$  and  a sequence of $t_i\to-\infty$ such that
\begin{align}\label{eq:inf limit of R}
\lim_{t\rightarrow-\infty}R(p,t_i)=C>0
\end{align}
and  rescaled $(M,R(p,t_i)g(R^{-1}(p,t_i)t),\phi_{t_{i}}(p))$ converges subsequently to  a  flow $(M_{\infty},g_{\infty}(t),p_{\infty})$ $(t\in (-\infty, \infty))$  of  complete $\kappa$-noncollapsed,  steady K\"{a}hler-Ricci solitons   
which splits  locally  off  a complex line.  Moreover
\begin{align}\label{eq:constant limit-2}
R_{\infty}(p_{\infty},t)=1,~for~t\in(-\infty,+\infty).
\end{align}
 Now we can consider    the  universal covering   $(\widetilde {M}_{\infty}, \widetilde g( 0))$  of  $({M}_{\infty}, g( 0))$.  As in  Corollary \ref{cor-of-splitting-theorem},   $(\widetilde {M}_{\infty}, \widetilde g( 0))$  can be split out  an $n-1$-dimensional   complete   $\kappa$-noncollapsed steady K\"{a}hler-Ricci soliton  $(N,g_N)$ with nonnegative bisectional curvature.   Note that  $(N,g_N)$  is not flat by (\ref{eq:constant limit-2}).   Thus by 
Cao's dimension reduction theorem, we may assume that $(N,g_N)$ has positive Ricci curvature.  On the other hand,  by the induction assumpation,  $(N,g_N)$ should be flat. This is a contradiction!  The claim is proved.

 By Claim \ref{key-lemma}, we can choose $p_{k}$ as in the proof of  Proposition \ref{flat-theorem} to apply Lemma  \ref{prop-contradiction  argument} to finish the proof of Theorem \ref{main-theorem-nonexistence-1} .

\end{proof}

 Theorem \ref{theorem-eternal-flow} is an application of Theorem  \ref{main-theorem-nonexistence-1} by using an argument  as in proof of  Theorem 1.4 in \cite{DZ2}.

\vskip6mm


\end{document}